\documentclass[a4paper]{amsart}
\usepackage{graphicx}
\usepackage{amssymb}
\usepackage{amsmath}
\usepackage{amsthm,amsfonts,bbm}
\usepackage{amscd}
\usepackage{geometry}
\usepackage[all,2cell]{xy}
\usepackage{epsfig,epstopdf}

\UseAllTwocells \SilentMatrices

\newtheorem{thm}{Theorem}[section]
\newtheorem{cor}[thm]{Corollary}
\newtheorem{lem}[thm]{Lemma}
\theoremstyle{definition}
\newtheorem{defi}[thm]{Definition}
\theoremstyle{remark}

\numberwithin{equation}{section}
\numberwithin{figure}{section}
\def\A{\mathcal{A}}

\def \e{\epsilon}
\def \E{\mathcal{E}}
\def \F{\mathcal{F}}
\def \H{\mathcal{H}}
\def\I{\mathcal{I}}
\def\la{\lambda}

\def \s{\sigma}

\def \T{\mathcal{T}}
\def \Tr{\text{Tr}}
\def \W{\mathbf{W}}

\begin{document}
\title[Trace and Estrada index]
{The trace and Estrada index of uniform hypergraphs with cut vertices}

\author[Y.-Z. Fan]{Yi-Zheng Fan*}
\address{Center for Pure Mathematics, School of Mathematical Sciences, Anhui University, Hefei 230601, P. R. China}
\email{fanyz@ahu.edu.cn}
\thanks{*The corresponding author.
This work was supported by National Natural Science Foundation of China (Grant No. 11871073).}

\author[Y. Yang]{Ya Yang}
\address{School of Mathematical Sciences, Anhui University, Hefei 230601, P. R. China}
\email{yangy@stu.ahu.edu.cn}

\author[C.-M. She]{Chuan-Ming She}
\address{School of Mathematical Sciences, Anhui University, Hefei 230601, P. R. China}
\email{cm-she@stu.ahu.edu.cn}

\author[J. Zheng]{Jian Zheng}
\address{School of Mathematical Sciences, Anhui University, Hefei 230601, P. R. China}
\email{zhengj@stu.ahu.edu.cn}

\author[Y.-M. Song]{Yi-Min Song}
\address{School of Mathematical Sciences, Anhui University, Hefei 230601, P. R. China}
\email{songym@stu.ahu.edu.cn}

\author[H.-X. Yang]{Hong-Xia Yang}
\address{School of Mathematical Sciences, Anhui University, Hefei 230601, P. R. China}
\email{yanghx@stu.ahu.edu.cn}



\begin{abstract}
Let $\H$ be an $m$-uniform hypergraph, and let $\A(\H)$ be the adjacency tensor of $\H$ which can be viewed as a system of homogeneous polynomials of degree $m-1$.
Morozov and Shakirov generalized the traces of linear systems to nonlinear homogeneous polynomial systems and obtained  explicit formulas for multidimensional resultants.
Sun, Zhou and Bu introduced the Estrada index of uniform hypergraphs which is closely related to the traces of their adjacency tensors.
In this paper we give formulas for the traces of $\A(\H)$ when $\H$ contains cut vertices, and obtain results on the traces and Estrada index when $\H$ is perturbed under local changes.
We prove that among all hypertrees with fixed number of edges, the hyperpath is the unique one with minimum Estrada index and the hyperstar is the unique one with maximum Estrada index.
\end{abstract}

\maketitle

\section{Introduction}
The traces of a square matrix are closely related to its characteristic polynomial by Newton identities.
A very famous identity
\begin{equation}\label{log}
\log \det(I-f)=\Tr\log(I-f)
\end{equation}
 gives the expression of determinant of $f$ by Schur polynomial in terms of the traces (see \cite{MS2011}), where $f: \mathbb{C}^n \to \mathbb{C}^n$ is a linear map, and $I$ is the identity map.
Morozov and Shakirov \cite{MS2011} genearlized the identity (\ref{log}) from determinants to resultants by introducing the traces of polynomial maps $f$ given by $n$ homogeneous polynomials of arbitrary degrees.

A \emph{tensor} (also called \emph{hypermatrix}) $\T=(t_{i_{1} i_2 \ldots i_{m}})$ of order $m$ and dimension $n$ over $\mathbb{C}$ refers to a
 multiarray of entries $t_{i_{1}i_2\ldots i_{m}}\in \mathbb{C}$ for all $i_{j}\in [n]:=\{1,2,\ldots,n\}$ and $j\in [m]$, which can be viewed to be the coordinates of the classical tensor (as a multilinear function) under a certain basis.
The tensor $\T$ is naturally associated with a (nonlinear) polynomial map $\varphi: \mathbb{C}^n \to \mathbb{C}^n$ given by
$$ \varphi_i(x_1,x_2,\ldots,x_n)=\sum_{i_2,\ldots,i_m}  t_{i i_2 \ldots i_{m}}x_{i_2}\cdots x_{i_m}, i \in [n].$$
Let $A=(a_{ij})$ be an $n \times n$ auxiliary matrix by thinking all $a_{ij}$'s as variables.
Using the traces defined by Morozov and Shakirov \cite{MS2011}, the $d$-th trace $\Tr_d(\T)$ of $\T$ (or $\Tr_d(\varphi)$)  is expressed as follow:
\begin{equation}\label{MSeq}\Tr_d(\T)=(m-1)^{(n-1)} \sum_{d_1+\cdots+d_n=d, \atop d_i \in \mathbb{N}, i \in [n]}
\prod_{i=1}^n \frac{1}{(d_i(m-1))!} \left(\sum_{y_i \in [n]^{m-1}}t_{iy_i}\frac{\partial }{\partial a_{iy_i}}\right)
\Tr(A^{d(m-1)}),
\end{equation}
where $t_{iy_i}=t_{ii_2 \ldots i_m}$ and $\frac{\partial }{\partial a_{iy_i}}=\frac{\partial }{\partial a_{ii_2}} \cdots \frac{\partial }{\partial a_{ii_m}}$ if
$y_i=(i_2 \ldots i_m)$.

Let $G$ be a simple graph and let $\A(G)$ be the adjacency matrix of $G$.
In this situation, the $d$-th trace of $\A(G)$ has a graph interpretation, namely,
$\Tr_d(\A(G))$ is the number of closed walks of length $d$ in the graph $G$.
Let $\H$ be an $m$-uniform hypergraph and let $\A(\H)$ be the adjacency tensor of $\H$.
Cooper and Dulte \cite{CD2012} applied Morozov-Shakirov's trace formula (\ref{MSeq}) to give some low co-degree coefficients of the characteristic polynomial of $\A(\H)$.
Shao, Qi and Hu \cite{SQH2015} gave a graph interpretion for the $d$-th trace of $\A(\H)$.
Recently Clark and Cooper \cite{CC2021} generalized the Harary-Sachs theorem of graphs to uniform hypergraphs by using Veblen hypergraphs.
Sun-Zhou-Bu \cite{SZB2021} generalized the Estrada index of graphs to uniform hypergraphs, which is closely related to the traces of the adjacency tensor.

\begin{defi}[\cite{SZB2021}]\label{defE}
Let $\H$ be an $m$-uniform hypergraph on $n$ vertices, and let $\la_1,\ldots,\la_N$ be all eigenvalues of the adjacency tensor $\A(\H)$ of $\H$, where $N=n(m-1)^{n-1}$.
The Estrada index of $\H$ is defined to be
\begin{equation}\label{EEgraph}
EE(G)=\sum_{i=1}^N e^{\lambda_i}.
\end{equation}
\end{defi}
As $\Tr_d(\A(\H))=\sum_{i=1}^N \la_i^d$, it is easily seen that
\begin{equation}\label{EE_tr} EE(G)=\sum_{d=0}^\infty \frac{\Tr_d(\A(G))}{d!}.\end{equation}
When $m=2$, the Estada index in Definition \ref{defE} is exactly that of a graph, 
which was first introduced by Estrada \cite{E2000} in 2000, which was found useful in measuring the degree of protein folding \cite{E2002} and the centrality of complex networks \cite{ER2005}.
Pe\~na, Gutman and Rada \cite{PGR2007} conjectured that
 the path is the unique graph with the minimum Estrada index among all graphs (trees) with given order, and the star is the unique one with the maximum Estrada index among all trees with given order.
 The conjecture was partly proved by Das and Lee \cite{DL2009}, and completely proved by Deng \cite{Deng2009}.
 Ili\'c and Stevanovi\'c \cite{IS2010} gave a new proof of the conjecture by considering the trees with fixed maximum degree.
 One can refer \cite{GDR2011} for a survey on Estrada index of graphs.
 We note here Lu et al. \cite{LXZ} introduced the Estrada index of uniform hypergraphs by the signless Laplacian matrix of the hypergraphs.

In this paper, we first give some perturbation results on the traces of a uniform hypergraph with cut vertices when the hypergraph is under some local changes, and then apply the results to show that among all hypertrees with fixed number of edges, the hyperpath is the unique one with minimum Estrada index and the hyperstar is the unique one with maximum Estrada index, which generalized the result of Pe\~na-Gutman-Rada's conjecture on simple graphs.

\section{Preliminaries}

\subsection{Tensors and hypergraphs}
Let $\T=(t_{i_{1} i_2 \ldots i_{m}})$ be a complex tensor of order $m$ and dimension $n$.
The tensor $\T$ is \emph{symmetric} if all entries $t_{i_1i_2\cdots i_m}$ are invariant under any permutation of its indices.
Given a vector $x\in \mathbb{C}^{n}$, $\T x^{m-1} \in \mathbb{C}^n$, which is defined as follows:
   \[
      (\T x^{m-1})_i =\sum_{i_{2},\ldots,i_{m}\in [n]}t_{ii_{2}\ldots i_{m}}x_{i_{2}}\cdots x_{i_m}, i \in [n].
  \]
 Let $\mathcal{I}=(i_{i_1i_2\ldots i_m})$ be the {\it identity tensor} of order $m$ and dimension $n$, that is, $i_{i_{1}i_2 \ldots i_{m}}=1$ if
   $i_{1}=i_2=\cdots=i_{m} \in [n]$ and $i_{i_{1}i_2 \ldots i_{m}}=0$ otherwise.
In 2005 Lim \cite{Lim} and Qi \cite{Qi} introduced the eigenvalues of tensors independently as follows.

\begin{defi}[\cite{Lim,Qi}]\label{eigen} Let $\T$ be an $m$-th order $n$-dimensional tensor.
For some $\lambda \in \mathbb{C}$, if the polynomial system $(\lambda \mathcal{I}-\T)x^{m-1}=0$, or equivalently $\T x^{m-1}=\lambda x^{[m-1]}$, has a solution $x\in \mathbb{C}^{n}\backslash \{0\}$,
then $\lambda $ is called an \emph{eigenvalue} of $\T$ and $x$ is an \emph{eigenvector} of $\T$ associated with $\lambda$,
where $x^{[m-1]}:=(x_1^{m-1}, x_2^{m-1},\ldots,x_n^{m-1})$.
\end{defi}

The \emph{determinant} of $\T$, denoted by $\det \T$, is defined to be the resultant of the polynomials $\T x^{m-1}$ \cite{Ha},
and the \emph{characteristic polynomial} of $\T$ is defined to be $\varphi_\T(\la):=\det(\la \I-\T)$ \cite{Qi,CPZ2}.
It is known that $\la$ is an eigenvalue of $\T$ if and only if it is a root of $\varphi_\T(\la)$.
The \emph{spectrum} of $\T$ is the multi-set of the roots of $\varphi_\T(\la)$.

A  \emph{hypergraph} $\H=(V,E)$ consists of a vertex set $V=\{v_1,v_2,{\cdots},v_n\}$ denoted by $V(\H)$ and an edge set $E=\{e_1,e_2,{\cdots},e_k\}$ denoted by $E(\H)$,
 where $e_i \subseteq V$ for $i \in [k]$.
 If $|e_i|=m$ for each $i \in [k]$ and $m \geq2$, then $\H$ is called an \emph{$m$-uniform} hypergraph.
 The \emph{degree} of a vertex $v$ in $\H$ is the number of edges of $\H$ containing the vertex $v$.
A vertex $v$ of $\H$ is called a \emph{cored vertex} if it has degree one.
An edge $e$ of $\H$ is called a \emph{pendent edge} if it contains $|e|-1$ cored vertices.
A {\it walk} $W$ in $\H$ is a sequence of alternate vertices and edges: $v_{0}e_{1}v_{1}e_{2}\cdots e_{l}v_{l}$,
    where $v_{i} \ne v_{i+1}$ and $\{v_{i},v_{i+1}\}\subseteq e_{i}$ for $i=0,1,\ldots,l-1$.
In the case of $\H$ being a simple graph, we simply write $W: v_0 v_1 \cdots v_l$ as each edge contains exactly two vertices.
The walk $W$ is called a {\it path} if no vertices or edges are repeated in the sequence.
If $v_0=v_l$, then $W$ is called a {\it circuit}, and is called a {\it cycle} if no vertices or edges are repeated except $v_0=v_l$.
The  hypergraph $\H$ is said to be {\it connected} if every two vertices are connected by a walk.
  The hypergraph $\H$ is called \emph{simple} if there exists no $i \ne j$ such that $e_i \subseteq e_j$.
In particular, a simple graph is a simple $2$-uniform hypergraph.
Throughout of this paper, \emph{all hypergraphs are considered connected, simple and $m$-uniform unless stated somewhere}.

Hu, Qi and Shao \cite{HQS} introduced a class of hypergraphs which are constructed from simple graphs.
Let $G=(V,E)$ be a simple graph. For any $m\geq3$, the {\it $m$-th power of $G$}, denoted by $G^{m}:=(V^{m},E^{m})$,
  is defined as the $m$-uniform hypergraph with the vertex set $V^{m}=V\cup{\{i_{e,1},\ldots,i_{e,m-2}: e\in E}\}$ and edge set
$E^{m}={\{e\cup{{\{i_{e,1},\ldots,i_{e,m-2}}}\}: e\in E}\}$.
 If a hypergraph is connected and acyclic, then it is called a {\it hypertree} (also called supertree in \cite{LSQ} and other literatures).
 Obviously, the power of trees is a special class of hypertrees.
 Denote $P_n$ and $S_n$ the path and cycle with $n$ edges respectively, both as simple graphs.
 The powers $P_n^m$ and $S_n^m$ are called \emph{hyperpath} and \emph{hyperstar} respectively.
 The notion of power of simple graphs were generalized by Khan and Fan \cite{KFan} and Kang et al. \cite{KLQY}.

In 2012 Cooper and Dutle \cite{CD2012} introduced the adjacency tensor of a uniform hypergraph, and applied the eigenvalues of the tensor to characterize the structural property of the hypergraph.

\begin{defi}[\cite{CD2012}]\label{def-adj}
Let $\H$ be an $m$-uniform hypergraph on $n$ vertices $v_1,v_2,\ldots,v_n$.
The \emph{adjacency tensor} of $\H$ is defined as $\mathcal{A}(\H)=(a_{i_{1}i_{2}\ldots i_{m}})$, an $m$-th order $n$-dimensional tensor, where
$$a_{i_1 i_2 \ldots i_m}=\left\{
\begin{array}{cl}
\frac{1}{(m-1)!}, & \mbox{~if~} \{v_{i_1},\ldots,v_{i_m}\} \in E(H);\\
0, & \mbox{~else}.
\end{array}\right.
$$
\end{defi}

Note that the adjacency tensor $\A(\H)$ is symmetric.
The spectrum and the traces of $\H$ are referring to those of $\A(\H)$.
Since the Perron-Frobenius theorem of nonnegative matrices was generalized to  nonnegative tensors \cite{CPZ1,FGH,YY1,YY2,YY3},
 the spectral hypergraph theory develops rapidly on many topics, such as the spectral radius \cite{BL,FanTPL, LSQ, LM}, the eigenvariety associated with the spectral radius \cite{FBH,FTL}, the spectral symmetry \cite{ SQH2015,zhou, FHB,FLW}, the relation between the eigenvalues and the roots of matching polynomial \cite{ZKSB}.

\subsection{Traces}
Let $\T=(t_{i_1\ldots i_m})$ be a tensor of order $m$ and dimension $n$.
Morozov and Shakirov defined the traces $\Tr_d(\T)$ as in (\ref{MSeq}).
Shao-Qi-Hu \cite{SQH2015} proved that
$$ \Tr_d(\T)=\sum_{i=1}^N \la_i^d,$$
which is consistent with the matrix case,
where $\la_1,\ldots,\la_N$ are all eigenvalues of $\T$, and $N=n(m-1)^{n-1}$.

%


Shao, Qi and Hu \cite{SQH2015} gave a graph interpretation for the $d$-th trace $\Tr_d(\T)$.
Let
$$\F_d=\{((i_1,\alpha_1),\ldots,(i_d,\alpha_d)): i_1 \le \cdots \le i_d, \alpha_j \in [n]^{m-1}\},$$
 where $i_j$ is called the \emph{primary index} (or \emph{root}) of the $m$-tuple
$f_j:=(i_j,\alpha_j)$ for $j \in [d]$.
Define an $i_j$-rooted directed star for the tuple $f_j$:
$$ S_{f_j}(i_j)=(V_j, \{(i_j,u_k): k=1,\ldots,m-1 \}),$$
where $V_j=\{i_j,u_1,\ldots,u_{m-1}\}$ as a set by omitting multiple indices if $\alpha_j=(u_1,\ldots,u_{m-1})$.
So we get a multi-directed graph associated with $F$, denoted and defined as
$$R(F)=\cup_{j=1}^d S_{f_j}(i_j).$$
Denote by $b(F)$ the product of the factorial of the multiplicities of all arcs of $R(F)$, $c(F)$ the product of the factorial of the outdegree of the all vertices of $R(F)$, and
$\W(F)$ the set of vertex sequences of all Euler tours of $R(F)$.
 Note that an \emph{Euler tour} of $R(F)$ means a sequence of alternative vertices and arcs
 \begin{equation}\label{euler}
 v_1 e_1 v_2 e_2 \cdots v_{k-1} e_{k-1} v_k
 \end{equation}
 such that it traverses each arc of $R(F)$ exactly once, where $v_k=v_1$ and $e_i =(v_i, v_{i+1})$ for $i \in [k-1]$.
  The \emph{vertex sequence of the Euler tour} (\ref{euler}) is obtained by deleting the arcs, namely, $v_1 v_2 \cdots v_k$.
  An \emph{Euler circuit} is an Euler tour up to cyclic permutation of its arcs, i.e., an Euler tour with no distinguished beginning arc.
Shao, Qi and Hu \cite{SQH2015} proved that
\begin{equation}\label{Shaoeq}\Tr_d(\T)=(m-1)^{(n-1)} \sum_{F \in \F_d} \frac{b(F)}{c(F)} \pi_F(\T) |\W(F)|,
\end{equation}
where $\pi_F(\T)=\prod_{i=1}^d t_{i_j,\alpha_j}$ if $F=((i_1,\alpha_1),\ldots,(i_d,\alpha_d))$.



Let $\H$ be an $m$-uniform hypergraph on $n$ vertices and let $\A(\H)$ be the adjacency tensor of $\H$.
We simply write $\Tr_d(\A(\H))$ as $\Tr_d(\H)$.
From Eq. (\ref{Shaoeq}), it suffices to consider the case  that each tuple in $F \in \F_d$ corresponds a nonzero entry of $\T$; otherwise $\pi_F(\T)$ is zero and has no contribution to the sum.
So, given an ordering of the vertices of $\H$, we only need to consider the following set instead of $\F_d$:
$$\F_d(\H):=\{(e_1(v_1),\ldots,e_d(v_d)): e_i \in E(\H), v_1 \le \cdots \le v_d\},$$
where $e_i(v_i)$ is an edge of $\H$ with root $v_i$.
Similarly we define rooted directed stars $S_{e_i}(v_i)$ and multi-directed graph
$R(F)=\cup_{i=1}^d S_{e_i}(v_i)$.
Furthermore, the directed graph $R(F)$ should contain Euler tours otherwise $|\W(F)|=0$.
So, it is enough to consider the following set
$$ \F^\e_d(\H):=\{F \in \F_d(\H): R(F) ~is ~Eulerian\}.$$
For each $F \in \F^\e_d(\H)$, the multi-hypergraph induced by the edges in $F$ by omitting the roots, denoted by $\mathcal{V}(F)$, is proved to be an $m$-valent $m$-uniform multi-hypergraph which is  called a \emph{Veblen hypergraph} \cite{CC2021}, where a hypergraph is called \emph{$m$-valent} if the degree of each vertex is a multiple of $m$.
Given a Veblen hypergraph $H$, a \emph{rooting} of $H$ is an ordering $F=(e_1(v_1),\ldots,e_t(v_t))$ of all edges of $H$, where $v_i$ is the root of $e_i$ for $i \in [t]$, and $v_1 \le \cdots \le v_t$ under the given order of the vertices of $H$.
If $R(F)$ is Eulerian, such rooting $F$ is called an \emph{Euler rooting} of $H$.

Now substituting $\T$ with the adjacency tensor $\A(\H)$ in (\ref{Shaoeq}), and noting that each $F \in \F_d^\e(\H)$ corresponds to $((m-1)!)^d$ ordered $d$-tuples in $\F_d$, we have
\begin{equation}\label{Shaoeq_Hyp1}\Tr_d(\H)=(m-1)^{(n-1)} \sum_{F \in \F_d^\e(\H)} \frac{b(F)}{c(F)} |\W(F)|,
\end{equation}
Denote by $\E(F)$ the set of Euler tours in $R(F)$,
and $\mathfrak{E}(F)$ the set of Euler circuits in $R(F)$.
Obviously,
\begin{equation}\label{tour} b(F)|\W(F)|=|\E(F)|=d(m-1)|\mathfrak{E}(F)|. 
\end{equation}
%
%
%
%

By BEST theorem \cite{TS1941, EB1987}, for a directed Euler graph $D$, fixed the beginning vertex say $u$,
\begin{equation}\label{cir}
|\mathfrak{E}(D)|=\tau_u(D) \prod_{v \in V(D)} (d^+(v)-1)!,
\end{equation}
where $\tau_u(D)$ is the number of arborescences of $D$ with root $u$ (namely, a directed
$u$-rooted spanning tree such that all vertices except $u$ has a directed path from itself to $u$), $d^+(v)$ denotes the outdegree of the vertex $v$.
For a general directed (not necessarily Euler) graph, the number $\tau_u(D)$ is the principal minor the Laplacian matrix  $L(D)$ of $D$ by deleting the row and column indexed by $u$, where
$$ L(D):=\text{diag}\{d^+(v): v \in V(D)\}-A(D),$$
and $A(D)$ is the adjacency matrix of $D$.
If $D$ is Eulerian,
then $ \tau_u(D)$ is independent of the choice of the root $u$, simply denoted by $\tau(D)$.
In particular, if denoting
$\overleftrightarrow{K_m}$ the directed graph obtained from the complete graph $K_m$ on $m$ vertices by replacing each edge by a pair of arcs with opposite directions, $t \overleftrightarrow{K_m}$ the $t$ copies of $\overleftrightarrow{K_m}$,
then
\begin{equation}\label{spantree}
\tau(t \overleftrightarrow{K_m})=t^{m-1}m^{m-2}.
\end{equation}

Combing Eqs. (\ref{tour}) and (\ref{cir}), we have
\begin{equation}\label{Shaoeq_Hyp3}\Tr_d(\H)=d(m-1)^{n} \sum_{F \in \F_d^\e(\H)} \frac{\tau(F)}{\prod_{v \in V(F)} d_F^+(v)}=:\sum_{F \in \F_d^\e(\H)} \omega(F),
\end{equation}
where $\tau(F):=\tau(R(F))$, $V(F):=V(R(F))$, $d_F^+(v)=(m-1)r_F(v)$ (outdegree of $u$ in $R(F)$), and $r_F(v)$ denotes the number of edges in $F$ with $v$ as the root.

For the convenience of discussion in next section, we also introduce some notions below.
For nonneagtive integers $i,j$, denote
$$ \F^\e_{d}(\H)[u_1,\ldots,u_i,\hat{v}_1,\ldots,\hat{v}_j]:=\{F \in \F_d^\e(\H): r_F(u_p)>0, r_F(v_q)=0, p \in [i], q \in [j]\},$$
$$Tr_d(\H;[u_1,\ldots,u_i,\hat{v}_1,\ldots,\hat{v}_j]):=\sum_{F \in \F^\e_{d}(\H)[u_1,\ldots,u_i,\hat{v}_1,\ldots,\hat{v}_j]} \omega(F).$$
For a positive integer $t \le d$, denote
$$\F^\e_{d;t}(\H)[u]:=\{F \in \F_d^\e(\H)[u]: r_F(u)=t\},
\Tr_{d;t}(\H;[u]):=\sum_{F \in \F_{d;t}^\e(\H)[u]} \omega(F).$$
Furthermore, denote
$$\F^\e_{d;t}(\H)[u;u_1,\ldots,u_i,\hat{v}_1,\ldots,\hat{v}_j]:=
\F^\e_{d;t}(\H)[u] \cap \F^\e_{d}(\H)[u,u_1,\ldots,u_i,\hat{v}_1,\ldots,\hat{v}_j],$$
$$Tr_{d;t}(\H;[u;u_1,\ldots,u_i,\hat{v}_1,\ldots,\hat{v}_j]):=\sum_{F \in \F^\e_{d;t}(\H)[u;u_1,\ldots,u_i,\hat{v}_1,\ldots,\hat{v}_j]} \omega(F).$$

\section{Perturbation of traces of uniform hypergraphs}
In this section, we mainly discuss the trace property of uniform hypergraphs with cut vertices.
A hypergraph is called \emph{nontrival} if it contains more than one vertex.
The \emph{coalescence} of two nontrivial connected hypergraphs $\H_1$ and $\H_2$ is obtained by identifying one vertex $v_1$ of $\H_1$ and one vertex $v_2$ of $\H_2$ and forming a new vertex $u$, denoted by $\H_1(v_1) \odot  \H_2(v_2)$, also written as $\H_1(u) \odot  \H_2(u)$.
Let $\H$ be a connected hypergraph.
A vertex $u \in V(\H)$ is called a cut vertex of $\H$   if $\H$ can be wrriten as $\H_1(u) \odot  \H_2(u)$, where  $\H_1,\H_2$ are both nontrivial and connected.
All hypergraphs are nontrivial, connected and $m$-uniform in this and the following sections.

Define $\frac{0}{0}:=1$ and $\Tr_{0;0}(\H):=(m-1)^{n-1}$ if $\H$ is an $m$-uniform hypergraph on $n$ vertices.

\begin{lem} \label{Tra_cut}
Let $\H=\H_1(u) \odot  \H_2(u)$ be a coalescence of $\H_1,\H_2$. Then
\begin{equation}\label{trace_cut}
    \Tr_d(\H_1(u) \odot  \H_2(u);[u])= \sum_{d_1+d_2=d}\sum_{t_1\le d_1, t_2\le d_2,\atop 0< t_1+t_2 } \frac{d}{t_1+t_2}{t_1+t_2 \choose t_1} \frac{t_1}{d_1}\Tr_{d_1;t_1}(\H_1;[u])  \frac{t_2}{d_2} \Tr_{d_2;t_2}(\H_2;[u]).
\end{equation}
\end{lem}

\begin{proof}
For each $F \in \F^\e_d(\H)[u]$, as $u$ is a cut vertex, each edge in $F$ is either contained in $\H_1$ or $\H_2$ but not both.
So, $F|_{\H_1}$, the restriction of $F$ onto $\H_1$, is an ordered rooted edges according their orders in $F$.
Similarly we have $F|_{\H_2}$, the restriction of $F$ onto $\H_2$.
Let
\begin{equation}\label{Fdd}
\F^\e_{d_1,d_2;t_1,t_2}(\H)[u]:=\{F \in \F^\e_d(\H)[u]: |F|_{\H_1}|=d_1, |F|_{\H_2}|=d_2, r_{F|_{\H_1}}(u)=t_1, r_{F|_{\H_2}}(u)=t_2,t_1+t_2>0\},
\end{equation}
where $|F|$ denotes the number of edges in $F$.
So we  have a decomposition of $ \F^\e_d(\H)[u]$ as follows:
$$ \F^\e_d(\H)[u]=\bigcup_{d_1+d_2=d}\bigcup_{t_1\le d_1,t_2\le d_2, \atop 0 < t_1+t_2 }\F^\e_{d_1,d_2;t_1,t_2}(\H)[u].$$
Note that there is a surjection from
$$\phi: \F^\e_{d_1,d_2;t_1,t_2}(\H)[u] \to \F^\e_{d_1;t_1}(\H_1)[u] \times \F^\e_{d_2;t_2}(\H_2)[u], ~ F \mapsto (F|_{\H_1},F|_{\H_2}),$$
and each pair $(F_1,F_2) \in \F^\e_{d_1;t_1}(\H_1)[u] \times \F^\e_{d_2;t_2}(\H_2)[u]$ has $t_1+t_2 \choose t_1$ preimages in   $\F^\e_{d_1,d_2;t_1,t_2}(\H)[u]$ as in each of the preimages $u$ is the root of $t_1$ edges in $\H_1$ and $t_2$ edges in $\H_2$.


Suppose $\H$ has $n$ vertices and $\H_1,\H_2$ have $n_1,n_2$ vertices respectively.
Then $n=n_1+n_2-1$.
By Eq. (\ref{Shaoeq_Hyp3}),
\begin{equation}\label{cut_tra_new}
\begin{split}
 \Tr_d(\H;[u])&=\sum_{F \in \F^\e_d(\H)[u]} \frac{d(m-1)^{n}\tau(F)}{\prod_{v \in V(F)} d_F^+(v)}\\
    &= \sum_{d_1+d_2=d}\sum_{t_1\le d_1, t_2\le d_2, \atop 0<t_1+t_2} \sum\limits_{F \in \F^\e_{d_1,d_2;t_1,t_2}(\H)[u]}
    \frac{d(m-1)^{n}}{(t_1+t_2)(m-1)} \prod_{i \in S_F}\frac{t_i (m-1)\tau(F|_{\H_i}) )}{\prod\limits_{v \in V(F|_{\H_i})} d_F^+(v)}\\
    & =\sum_{d_1+d_2=d}\sum_{t_1\le d_1, t_2\le d_2,\atop 0< t_1+t_2 } {t_1+t_2 \choose t_1} \frac{d(m-1)^{n-1-\sum_{i \in S_F}(n_i-1)}}{(t_1+t_2)}\prod_{i \in S_F}\frac{t_i}{d_i}
    \sum_{F \in \F^\e_{d_i;t_i}(\H_i)[u]} \frac{d_i(m-1)^{n_i}\tau(F)}{\prod\limits_{v \in V(F)} d_F^+(v)}
    \\
    &= \sum_{d_1+d_2=d}\sum_{t_1\le d_1, t_2\le d_2,\atop 0< t_1+t_2 } \frac{d}{t_1+t_2}{t_1+t_2 \choose t_1} \frac{t_1}{d_1}\Tr_{d_1;t_1}(\H_1;[u])  \frac{t_2}{d_2} \Tr_{d_2;t_2}(\H_2;[u]),
\end{split}
\end{equation}
where $S_F:=\{i: i \in [2], t_i \ne 0\}$.
Note that $d_i=0$ if and only $t_i=0$, and in this case we have $\frac{t_i}{d_i}=1$ and
$\Tr_{d_i;t_i}(\H_i;[u])=(m-1)^{n_i-1}$.
 So that the Eq. (\ref{cut_tra_new}) holds.
\end{proof}

\begin{cor}\label{Tra_cut_2}
Let $\H_{uw}=\H_1(u) \odot  \H_2(w)$ and $\H_{vw}=\H_1(v) \odot  \H_2(w)$.
For all positive integers $d_1,t_2$ such that
if
\begin{equation}\label{Tr_uv}
\sum_{0< t_1\le d_1} {{t_1+t_2-1} \choose t_2} \Tr_{d_1;t_1}(\H_1;[u]) \ge \sum_{0<t_1\le d_1} {{t_1+t_2-1} \choose t_2} \Tr_{d_1;t_1}(\H_1;[v]),\end{equation}
or equivalently,
\begin{equation}\label{Tr_uv_an}
\sum_{F \in \F^\e_{d_1}(\H_1)[u]} {{r_F(u)+t_2-1} \choose t_2} \frac{\tau(F)}{\prod_{x \in V(F)}d_F^+(x)} \ge \sum_{F \in \F^\e_{d_1}(\H_1)[v]} {{r_F(v)+t_2-1} \choose t_2} \frac{\tau(F)}{\prod_{x \in V(F)}d_F^+(x)},
\end{equation}
then $\Tr_d(\H_{uw}) \ge \Tr_d(\H_{vw})$ for all positive integers $d>d_1$.

Furthermore, if Eq. (\ref{Tr_uv}) or (\ref{Tr_uv_an}) is strict for some $d_1$ and $t_2$,
then $\Tr_d(\H_{uw}) > \Tr_d(\H_{vw})$ for all positive integers $d > d_1$.
\end{cor}

\begin{proof}
We have a decomposition
$$ \F^\e_d(\H_{uw})=\F^\e_d(\H_1) \cup \F^\e_d(\H_2) \cup \tilde{\F}_d(\H_{uw})[u,w],$$
where
 \begin{equation}\label{Huw}
 \tilde{\F}_d(\H_{uw})[u,w]:=\{F \in \F^\e_d(\H_{uw}): r_{F|_{\H_1}}(u)>0,
 r_{F|_{\H_2}}(w)>0\},
 \end{equation}
 So,
$$\Tr_d(\H_{uw})=\sum_{F \in \F^\e_d(\H_1)}\omega(F)+\sum_{F \in \F^\e_d(\H_2)}\omega(F)+\sum_{F \in \tilde{\F}_d(\H_{uw})[u,w]}\omega(F).$$
Similarly,
$$\Tr_d(\H_{vw})=\sum_{F \in \F^\e_d(\H_1)}\omega(F)+\sum_{F \in \F^\e_d(\H_2)}\omega(F)+\sum_{F \in \tilde{\F}_d(\H_{vw})[v,w]}\omega(F),$$
where $\tilde{\F}_d(\H_{vw})[v,w]$ is defined as in (\ref{Huw}) by replacing $u$ with $v$.
So,
$$\Tr_d(\H_{uw})-\Tr_d(\H_{vw})=\sum_{F \in \tilde{\F}_d(\H_{uw})[u,w]}\omega(F)-
\sum_{F \in \tilde{\F}_d(\H_{vw})[v,w]}\omega(F).$$

By Lemmas \ref{Tra_cut},
\begin{equation}\label{tracesum}
\sum_{F \in \tilde{\F}_d(\H_{uw})[u,w]}\omega(F)=
 \sum_{d_1+d_2=d,\atop 0< d_1,0< d_2}\sum_{0< t_2\le d_2}\frac{d t_2}{d_1d_2}\Tr_{d_2;t_2}(\H_2;[w])\sum_{0<t_1 \le d_1}{{t_1+t_2-1} \choose t_2} \Tr_{d_1;t_1}(\H_1;[u]),
\end{equation}
\begin{equation}\label{tracesum2}\sum_{F \in \tilde{\F}_d(\H_{uv})[v,w]}\omega(F)=
 \sum_{d_1+d_2=d,\atop 0< d_1,0< d_2}\sum_{0< t_2\le d_2}\frac{dt_2}{d_1d_2}\Tr_{d_2;t_2}(\H_2;[w])\sum_{0< t_1\le d_1}{{t_1+t_2-1} \choose t_2} \Tr_{d_1;t_1}(\H_1;[v]).
\end{equation}
So, if Eq. (\ref{Tr_uv}) holds, then by Eqs. (\ref{tracesum}) and (\ref{tracesum2}), we have $\Tr_d(\H_{uw}) \ge \Tr_d(\H_{vw})$ for all positive integers $d>d_1$.
If Eq. (\ref{Tr_uv}) is strict for some $d_1$ and $t_2$,
surely $\Tr_d(\H_{uw}) > \Tr_d(\H_{vw})$ for all positive integers $d>d_1$.

Suppose $\H_1$ has $n_1$ vertices.
By Eq. (\ref{Shaoeq_Hyp3}),
\begin{align*}
\sum_{0<t_1\le d_1} {{t_1+t_2-1} \choose t_2} \Tr_{d_1;t_1}(\H_1;[u])
&=d_1(m-1)^{n_1} \sum_{0< t_1\le d_1} {{t_1+t_2-1} \choose t_2}
\sum_{F \in \F^\e_{d_1;t_1}(\H_1)[u]}\frac{\tau(F)}{\prod_{x \in V(F)}d_F^+(x)}\\
&=d_1(m-1)^{n_1} \sum_{F \in \F^\e_{d_1}(\H_1)[u]} {{r_F(u)+t_2-1} \choose t_2}
\frac{\tau(F)}{\prod_{x \in V(F)}d_F^+(x)}.
\end{align*}
Similarly,
$$\sum_{0<t_1\le d_1} {{t_1+t_2-1} \choose t_2} \Tr_{d_1;t_1}(\H_1;[v])=
d_1(m-1)^{n_1} \sum_{F \in \F^\e_{d_1}(\H_1)[v]} {{r_F(v)+t_2-1} \choose t_2}
\frac{\tau(F)}{\prod_{x \in V(F)}d_F^+(x)}.$$
So Eq. (\ref{Tr_uv}) is equivalent to Eq. (\ref{Tr_uv_an}).
The result follows.
\end{proof}


\begin{lem}\label{core}
Let $\H$ be an $m$-uniform Veblen multi-hypergraph, and let $e$ be an edge of $\H$  which contains a cored vertex.
If $\H$ has an Euler rooting, then $e$ repeats $k \cdot m$ times for some positive integer $k$, and all cored vertices in $e$ occur as a root of $e$ in $k$ times.
\end{lem}

\begin{proof}
Suppose that $e$ repeats $t$ times in $\H$.
Let $F$ be an Euler rooting of $\H$.
As $R(F)$ is Eulerian, any vertex of $e$ will occur as a root.
Suppose that $v \in e$ is cored vertex which occurs as root in $k$ times for some positive integer $k$.
Noting that $v$ is only contained in $e$, so $v$ occurs as a non-root in $t-k$ times.
Since $R(F)$ is Eulerian, the outdegree of $v$ in $R(F)$ is same as its indegree, which yields that $ k(m-1)=t-k.$
So $t=km$.
If $u \in e$ is another cored vertex which occurs as root in $k'$ times, then
$k'(m-1)=km-k'$, which implies that $k'=k$.
The result follows.
\end{proof}

The underlying hypergraph of a multi-hypergraph $\H$, denoted by $\underline{\H}$, is the hypergraph obtained from $\H$ by removing its duplicate edges.

\begin{cor}\label{tree_root}
Let $\H$ be an $m$-uniform Veblen multi-hypergraph whose underlying hypergraph $\underline{\H}$ is a hypertree.
Then $\H$ has a unique Euler rooting such that all vertices of each edge occur as roots in same times, and hence every edge of $\H$ repeats in a multiple of $m$ times.
\end{cor}

\begin{proof}

Let $\eta(\underline{\H})$ be the number of edges of $\underline{\H}$.
We will show the assertion by induction on $\eta(\underline{\H})$.
If $\eta(\underline{\H})=1$, the assertion holds by Lemma \ref{core}

Assume the assertion holds for $\eta(\underline{\H})=t-1$.
Suppose $\underline{\H}$ has $t$ edges.
As $\underline{\H}$ is a hypertree, it contains at least one pendent edge $e$.
By Lemma \ref{core}, $\H$ contains $km$ copies of $e$, denoted by $e^{(km)}$ for some positive integer $k$. 
Let $\H'$ be obtained from $\H$ by deleting the edges $e^{(km)}$ together with the vertices therein.
Note that $\H'$ is also an $m$-uniform Veblen multi-hypergraph whose underlying hypergraph $\underline{\H'}$ is a hypertree with $t-1$ edges.
By the assumption,
$\H'$ has a unique Euler rooting $F'$ such that all vertices of each edge occur as roots in same times.
Obviously, $e^{(km)}$ has a unique Euler rooting $F''$ with all vertices occurring as root in $k$ times.
So, combing $F'$ and $F''$  we will get an Euler rooting of $\H$.

Now let $F$ be an Euler rooting of $\H$.
Then $F|_{\H'}$ and $F|_{e^{(km)}}$ are respectively the Euler rootings $\H'$ and $e^{(km)}$.
By the above discussion, the Euler rootings of $\H'$ and $e^{(km)}$ are both unique.
So $\H$ has a unique Euler rooting with property as in the assertion.
The result follows by induction.
\end{proof}

Let $\H_0,\H_1,\ldots,\H_p$ be pairwise disjoint connected hypergraphs, where $p \ge 1$.
Let $w_0 \in V(\H_0)$, and $u_i \in V(\H_i)$ for each $i \in [p]$.
Denote by $\H_0(w_0) \odot (\H_1(u_1),\ldots,\H_p(u_p))$ the hypergraph obtained from $\H_0$ by attaching $\H_1,\ldots,\H_p$ to $\H_0$ with $u_1,\ldots,u_p$ all identified with $w_0$.
Let $w_1,\ldots,w_p \in V(\H_0)$.
Denote by $\H_0(w_1,\ldots,w_p) \odot (\H_1(u_1),\ldots,\H_p(u_p))$ the hypergraph obtained from $\H_0$ by attaching $\H_1,\ldots,\H_p$ to $\H_0$ with $u_i$ identified with $w_i$ for each $i \in [p]$.

\begin{cor}\label{comp_path}
Let $\H_{r,s}:=\H(w) \odot (P_{r}^m(u_1), P_s^m(u_2))$, where $u_1,u_2$ are respectively the pendent vertices of $P_{r}^m$ and $P_s^m$.
If $r \ge s \ge 1$, then
$$ \Tr_d(\H_{r,s}) \ge \Tr_d(\H_{r+1,s-1}),$$
with strict inequality if $d \ge sm$.
\end{cor}

\begin{proof}
We label and order the vertices of $P_{r+s}^m$ from left to right as follows (see Fig. \ref{path}):
$$  v_{s-1}<v_{s-1,1}<\cdots<v_{s-1,m-2}<v_{s-2}<\cdots<v_1<\cdots< v_0<\cdots<u_0<\cdots<u_1<\cdots<u_{r},$$
and label the edges of $P_{r+s}^m$ from left to right as
$$ e_{s-1}, \cdots, e_1, e_0, e'_1,\cdots, e'_r,$$
where $e_i=\{v_i,v_{i,1},\ldots,v_{i,m-2}, v_{i-1}\}$ for $i=0,1,\ldots,s-1$, $v_{-1}:=u_0$, and $e'_i=\{u_{i-1}, u_{i-1,1},\ldots,u_{i-1,m-2},u_i\}$ for $i=1,2,\ldots,r$.
From Fig. \ref{path}, $P_{r+s}^m$ is obtained from the path $P_{r+s}: v_{s-1}v_{s-2}\cdots v_1 v_0 u_0 u_1 \cdots u_r$ by inserting $m-2$ additional vertices into each edge.

\vspace{2mm}

\begin{figure}[h]
\centering
\includegraphics[scale=.8]{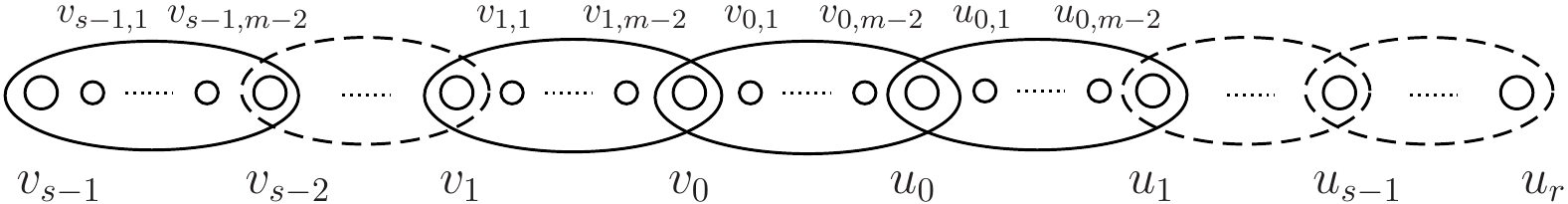}
\caption{The hyperpath $P_{r+s}^m$}\label{path}
\end{figure}

Let $u:=u_0$ and $v:=v_0$. Thus
$\H_{r,s}=\H(w)\odot P_{r+s}^m(u)$ and $\H_{r+1,s-1}=\H(w)\odot P_{r+s}^m(v)$.
We will prove that Eq. (\ref{Tr_uv}) holds for $\H_1=P_{r+s}^m=:P$.
Define a ``reflection'' $\phi$ on $P$ with respect to the edge $e_0$, that is,
$$ \phi=\prod_{i=0}^{s-1}(v_i u_i) \cdot \prod_{i=1}^{s-1} \prod_{j=1}^{ m-2 }(v_{i,j} u_{i-1,j}),$$
where $(ab)$ denotes the transposition that swaps $a$ and $b$.
Considering the restriction of $\phi$ on the edges of $P$, $\phi$ swaps $e_i$ and $e'_i$ for $i \in [s-1]$ and fixed all other edges.

Note that
$$ \F^\e_{d_1}(P)[v]=\F^\e_{d_1}(P)[v,\hat{u}_s] \cup \F^\e_{d_1}(P)[v,u_s],~
\F^\e_{d_1}(P)[u]=\F^\e_{d_1}(P)[u,\hat{u}_s] \cup \F^\e_{d_1}(P)[u,u_s],
$$
For each $F \in \F^\e_{d_1}(P)[v,\hat{u}_s]$, as the Veblen hypergraph $\mathcal{V}(F)$ is connected, $\underline{\mathcal{V}(F)}$ is a hyperpath which contains no the edge $e'_s$.
Define a map from $\F^\e_{d_1}(P)[v,\hat{u}_s]$ to $\F^\e_{d_1}(P)[u,\hat{u}_s]$, denoted by $\tilde{\phi}$, such that if
$ F=(e_{q_1}(v_{q_1}), \ldots, e_{0}(v_0), \ldots, e_{0}(u_0), \ldots, e'_{q_2}(u_{q_2}))$, then
\begin{equation}\label{tphi}\tilde{\phi}(F)=(\phi(e'_{q_2})(\phi(u_{q_2})), \ldots, \phi(e_{0})(\phi(u_0)), \ldots, \phi(e_{0})(\phi(v_0)), \phi(e_{q_1})(\phi(v_{q_1}))).
\end{equation}
By Corollary \ref{tree_root}, $r_F(v)=r_{\tilde{\phi}(F)}(u)$, and $\tilde{\phi}$ is a bijection between $\F^\e_{d_1;t_1}(P)[v;\hat{u}_s]$ and $\F^\e_{d_1;t_1}(P)[u;\hat{u}_s]$.
By Eq. (\ref{Shaoeq_Hyp3}), we have
\begin{equation}\label{refl}
\Tr_{d_1;t_1}(P;[v;\hat{u}_s])=\Tr_{d_1;t_1}(P;[u;\hat{u}_s]).
\end{equation}

Note that
$\F^\e_{d_1}(P)[v,{u}_s]=\F^\e_{d_1}(P)[v,u,{u}_s]$ and
\begin{equation}\label{addi} \F^\e_{d_1}(P)[u,{u}_s]=\F^\e_{d_1}(P)[u,{u}_s,\hat{v}]\cup \F^\e_{d_1}(P)[u,{u}_s, v].
\end{equation}
By Corollary \ref{tree_root}, $\F^\e_{d_1}(P)[u,{u}_s,\hat{v}] \ne \emptyset $  if $d_1 \ge sm$,
and $\F^\e_{d_1}(P)[u,{u}_s, v]\ne \emptyset $  if $d_1 \ge (s+1)m$.
We will prove that if $d_1 \ge (s+1)m$,
\begin{equation}\label{comp}
\sum_{0< t_1\le d_1} {{t_1+t_2-1} \choose t_2} \Tr_{d_1;t_1}(P;[u;v,u_s]) \ge \sum_{0< t_1\le d_1} {{t_1+t_2-1} \choose t_2} \Tr_{d_1;t_1}(P;[v;u,u_s]),
\end{equation}
or equivalently
\begin{equation}\label{compequiv}
\sum_{F \in \F^\e_{d_1}(P)[u,v,{u}_s]} {{r_F(u)+t_2-1} \choose t_2} \frac{\tau(F)}{\prod_{x \in V(F)}d_F^+(x)}
\ge \sum_{F^\e \in \F_{d_1}(P)[u,v,{u}_s]} {{r_F(v)+t_2-1} \choose t_2} \frac{\tau(F)}{\prod_{x \in V(F)}d_F^+(x)}.
\end{equation}

By Corollary \ref{tree_root}, if $ \F^\e_{d_1}(P)[u,v,{u}_s] \ne \emptyset$, then $m \mid d_1$; furthermore, we have a decomposition as
$$ \F^\e_{d_1}(P)[u,v,{u}_s]=\bigcup_{0\le q_v \le s-1,\atop s \le q_u \le r}
\bigcup_{\sum\limits_{i=0}^{q_v} m_i+ \sum\limits_{i=1}^{q_u} m_i=\frac{d_1}{m}}\bigcup_{\sigma_i \in \mathbb{S}(\{m_i,m'_i\}), \atop
i \in [q_v]}\F(m_0,\sigma_1(m_1,m'_1)\ldots,
\sigma_{q_v}(m_{q_v},m'_{q_v}),m'_{q_v+1},\ldots,m'_{q_u})
,$$
where
$\F(m_0,\sigma_1(m_1,m'_1)\ldots,
\sigma_{q_v}(m_{q_v},m'_{q_v}),m'_{q_v+1},\ldots,m'_{q_u})$ denotes the set of $d_1$-tuples of ordered rooted edges consisting of $e_0$ with multiplicity $mm_0$, $e_i$ with multiplicity $m\sigma_i(m_i)$,  $e'_i$ with multiplicity $m\sigma_i(m'_i)$ for $i \in [q_v]$, and edges $e'_{j}$ with multiplicity $mm'_j$ for $j=q_v+1,\ldots,q_u$,
and $\mathbb{S}(\{m_i,m'_i\})$ denotes the symmetric group on the set $\{m_i,m'_i\}$ for $i \in [q_v]$.

Given $q_v,q_u,m_0,m_1,m'_1,\ldots,m_{q_v},m'_{q_v},m'_{q_v+1},\ldots,m'_{q_u}$, simply denote
$$\F(\sigma_1,\sigma_2,\ldots,\sigma_{q_v}):=\F(m_0,\sigma_1(m_1,m'_1)\ldots,
\sigma_{q_v}(m_{q_v},m'_{q_v}),m'_{q_v+1},\ldots,m'_{q_u}),$$
and
$$
D(\sigma_1,\sigma_2,\ldots,\sigma_{q_v}):=
\sum_{F\in \F(\sigma_1,\sigma_2,\ldots,\sigma_{q_v})}
  \left({{r_F(u)+t_2-1} \choose t_2}
- {{r_F(v)+t_2-1} \choose t_2} \right)\frac{\tau(F)}{\prod\limits_{x \in V(F)}d_F^+(x)}.
$$
For $2 \le i \le q_v$, denote
$$ D(\sigma_i,\ldots,\sigma_{q_v})
=\sum_{\sigma_1,\ldots,\sigma_{i-1}}D(\sigma_1,\sigma_2,\ldots,\sigma_{q_v}),
$$
and
$$ D_{\emptyset}=\sum_{\sigma_1,\ldots,\sigma_{q_v}}
D(\sigma_1,\sigma_2,\ldots,\sigma_{q_v}).
$$
We may assume that $m_i \ne m'_i$ for each $i \in [q_v]$;
otherwise, $D_\emptyset =0$ as shown in the below.

Note that for $i=0,1, \ldots,q_v-1$, $v_i$ is the root of two adjancent multiple edges,  and
for $i=0,1,\ldots q_u-1$, $u_i$ is the root of two adjancent multiple edges, which causes different elements in $\F(\sigma_1,\sigma_2,\ldots,\sigma_{q_v})$.
 In fact, by Corollary \ref{tree_root} there are exactly
 $$\prod_{i=0}^{q_{v}-1}{\sigma_i(m_i)+\sigma_{i+1}(m_{i+1}) \choose \sigma_i(m_i)}{\sigma_i(m'_i)+\sigma_{i+1}(m'_{i+1}) \choose \sigma_i(m'_i)}\cdot
 {\sigma_{q_v}(m'_{q_v})+m'_{q_v+1} \choose \sigma_{q_v}(m'_{q_v})} \cdot
 \prod_{i=q_v+1}^{q_u-1} {m'_{i}+m'_{i+1} \choose m'_i}
 $$
 $d_1$-tuples of ordered rooted edges in $\F(\sigma_1,\ldots,\sigma_{q_v})$, all corresponding to the same summand in $D(\sigma_1,\ldots,\sigma_{q_v})$, where $\s_0=1$ and $m_0=m'_0$.
By Corollary \ref{tree_root}, for each $F\in \F(\sigma_1,\sigma_2,\ldots,\sigma_{q_v})$,
$$ r_F(u)=m_0+\sigma_1(m'_1),
r_F(v)=m_0+\sigma_1(m_1),$$
\begin{align*}
\prod_{x \in V(F)}d_F^+(x)
&=\sigma_{q_v}(m_{q_v})(\sigma_{q_v}(m'_{q_v})+m'_{q_v+1})\prod_{i=0}^{q_v-1} (\sigma_i(m_i)+\sigma_{i+1}(m_{i+1}))(\sigma_i(m'_i)+\sigma_{i+1}(m'_{i+1}))\\
& \cdot \prod_{i=q_v+1}^{q_u-1}(m'_i+m'_{i+1}) \cdot m'_{q_u}
\left(m_0\prod_{i=1}^{q_v} m_i m'_i \prod_{i=q_v+1}^{q_u} m'_i\right)^{m-2}(m-1)^{(q_u+q_v+1)(m-1)+1},
\end{align*}
and by Eq. (\ref{spantree}),
$$\tau(F)=m^{(q_u+q_v+1)(m-2)}\left(m_0\prod_{i=1}^{q_v} m_i m'_i \prod_{i=q_v+1}^{q_u} m'_i\right)^{m-1}.$$
So we express $D(\sigma_1,\sigma_2,\ldots,\sigma_{q_v})$ as follows:
\begin{align*}
D(\sigma_1,\sigma_2,\ldots,\sigma_{q_v})&=\alpha {m_0+\sigma_{1}(m_{1})-1 \choose m_0-1}{m_0+\sigma_{1}(m'_{1})-1 \choose m_0-1}\\
& \quad \cdot \left(
{m_0+\sigma_1(m'_1)+t_2-1 \choose t_2} -
{m_0+\sigma_1(m_1)+t_2-1 \choose t_2}\right)\\
&\quad \cdot {\sigma_{q_v}(m'_{q_v})+m'_{q_v+1}-1 \choose \sigma_{q_v}(m'_{q_v})-1} \prod_{i=1}^{q_{v}-1}{\sigma_i(m_i)+\sigma_{i+1}(m_{i+1})-1 \choose \sigma_i(m_i)-1}{\sigma_i(m'_i)+\sigma_{i+1}(m'_{i+1})-1 \choose \sigma_i(m'_i)-1} \cdot,
 \end{align*}
where $$\alpha=\frac{m^{(q_u+q_v+1)(m-2)} \prod_{i=q_v+1}^{q_u-1} {m'_{i}+m'_{i+1} -1\choose m'_i-1}}
{m_0(m-1)^{(q_u+q_v+1)(m-1)+1} }
.$$

Denote
$$f_1(m_0,m_1,m'_1)=
\frac{{m_0+m_1-1 \choose m_0-1}{m_0+m'_1-1 \choose m_0-1}\left({m_0+m'_1+t_2-1 \choose t_2}-
{m_0+m_1+t_2-1 \choose t_2}\right)}{m'_1-m_1},$$
and
$$f_i(m_{i-1},m'_{i-1},m_i,m'_i)=\frac{{m_{i-1}+m_i-1 \choose m_{i-1}-1}{m'_{i-1}+m'_i-1 \choose m'_{i-1}-1}-
{m'_{i-1}+m_i-1 \choose m'_{i-1}-1}{m_{i-1}+m'_i-1 \choose m_{i-1}-1}}{(m'_{i-1}-m_{i-1})(m'_i-m_i)},i=2,\ldots,q_v.$$
We have
\begin{align*}
D(\sigma_2,\ldots,\sigma_{q_v})&=D(1,\sigma_2,\ldots,\sigma_{q_v})+
D((m_1m'_1)),\sigma_2,\ldots,\sigma_{q_v})\\
&=\alpha f_1(m_0,m_1,m'_1)(m'_1-m_1)\\
&\cdot
\left({m_1+\sigma_2(m_2)-1 \choose m_1-1}{m'_1+\sigma_2(m'_2)-1 \choose m'_1-1}
-{m'_1+\sigma_2(m_2)-1 \choose m'_1-1}{m_1+\sigma_2(m'_2)-1 \choose m_1-1}\right)
\\
&\cdot {\sigma_{q_v}(m'_{q_v})+m'_{q_v+1}-1 \choose \sigma_{q_v}(m'_{q_v})-1} \prod_{i=2}^{q_{v}-1}{\sigma_i(m_i)+\sigma_{i+1}(m_{i+1}-1) \choose \sigma_i(m_i)-1}{\sigma_i(m'_i)+\sigma_{i+1}(m'_{i+1})-1 \choose \sigma_i(m'_i)-1}.
\end{align*}
Similarly,
\begin{align*}
D(\sigma_3,\ldots,\sigma_{q_v})&=\alpha f_1(m_0,m_1,m'_1)f_2(m_1,m'_1,m_2,m'_2)(m'_1-m_1)^2(m'_2-m_2)\\
&\quad \cdot
\left({m_2+\sigma_3(m_3)-1 \choose m_2-1}{m'_2+\sigma_3(m'_3)-1 \choose m'_2-1}
-{m'_2+\sigma_3(m_3)-1 \choose m'_2-1}{m_2+\sigma_3(m'_3)-1 \choose m_2-1}\right)
\\
& \quad \cdot {\sigma_{q_v}(m'_{q_v})+m'_{q_v+1}-1 \choose \sigma_{q_v}(m'_{q_v})-1}\prod_{i=3}^{q_{v}-1}{\sigma_i(m_i)+\sigma_{i+1}(m_{i+1})-1 \choose \sigma_i(m_i)-1}{\sigma_i(m'_i)+\sigma_{i+1}(m'_{i+1})-1 \choose \sigma_i(m'_i)-1}.
\end{align*}

By induction, we have
\begin{align*}
D(\sigma_{q_v})&=\alpha f_1(m_0,m_1,m'_1) \prod_{i=2}^{q_v-1}f_i(m_{i-1},m'_{i-1},m_i,m'_i)\cdot
\prod_{i=1}^{q_v-2}(m'_i-m_i)^2\cdot(m'_{q_v-1}-m_{q_v-1})\\
&\quad \cdot
\left({m_{q_v-1}+\sigma_{q_v}(m_{q_v})-1 \choose m_{q_v-1}-1}{m'_{q_v-1}+\sigma_{q_v}(m'_{q_v})-1 \choose m_{q'_v-1}-1}
-{m'_{q_v-1}+\sigma_{q_v}(m_{q_v})-1 \choose m'_{q_v-1}-1}{m_{q_v-1}+\sigma_{q_v}(m'_{q_v})-1 \choose m_{q_v-1}-1}\right)
\\
& \quad \cdot {\sigma_{q_v}(m'_{q_v})+m'_{q_v+1}-1 \choose \sigma_{q_v}(m'_{q_v})-1}.
\end{align*}
So,
\begin{align*}
D_{\emptyset}&=D(1)+D((m_{q}m'_{q})\\
&=\alpha f_1(m_0,m_1,m'_1)\cdot \prod_{i=2}^{q_v}f_i(m_{i-1},m'_{i-1},m_i,m'_i)\cdot
\prod_{i=1}^{q_v}(m'_i-m_i)^2\\
&\cdot
\frac{1}{(m'_{q_v}-m_{q_v})}\left({m'_{q_v}+m'_{q_v+1}-1 \choose m'_{q_v}-1}-{m_{q_v}+m'_{q_v+1}-1 \choose m_{q_v}-1}\right)
\end{align*}

In the above discussion, if there exists a smallest $i \in [q_v]$ such that
$m_i=m'_i$, then $D(\sigma_i,\ldots,\sigma_{q_v})=0$ and hence
$$D_\emptyset=\sum_{\sigma_i,\ldots,\sigma_{q_v}}D(\sigma_i,\ldots,\sigma_{q_v})=0.$$
So, if $m_i\ne m'_i$ for all $i\in [q_v]$, then $f_i(m_{i-1},m'_{i-1},m_i,m'_i)>0$ for all $i$, implying that
$ D_\emptyset >0.$
We are now arriving at the inequality (\ref{comp}) as by Eq. (\ref{compequiv}) and the decomposition of $\F^\e_{d_1}(P)[u,v,{u}_s]$,
\begin{align*}
&\sum_{0< t_1\le d_1} {{t_1+t_2}-1 \choose t_2} \Tr_{d_1;t_1}(P;[u;v,u_s]) - \sum_{0< t_1\le d_1} {{t_1+t_2}-1 \choose t_2} \Tr_{d_1;t_1}(P;[u;v,u_s])\\
&\quad =\sum_{0 \le q_v \le s-1,\atop s \le q_u \le r}
\sum_{\sum_{i=0}^{q_v} m_i +\sum_{i=1}^{q_u} m'_i=\frac{d_1}{m}}\sum_{\sigma_i \in \mathbb{S}(\{m_i,m'_i\}),i\in [q_v]}D(\sigma_1,\ldots,\sigma_{q_v})\\
& \quad \ge 0.
\end{align*}

By Eqs. (\ref{refl}), (\ref{addi}) and (\ref{comp}),
we proved that Eq. (\ref{Tr_uv}) holds for $\H_1=P$ as
\begin{align*}
&\sum_{0< t_1\le d_1} {{t_1+t_2}-1 \choose t_2} \Tr_{d_1;t_1}(P;[u]) - \sum_{0< t_1\le d_1} {{t_1+t_2}-1 \choose t_2} \Tr_{d_1;t_1}(P;[v])\\
&\quad =\sum_{0< t_1\le d_1} {{t_1+t_2}-1 \choose t_2} \Tr_{d_1;t_1}(P;[u;u_s, \hat{v}])\\
&\quad \quad +\sum_{0< t_1\le d_1} {{t_1+t_2}-1 \choose t_2} \Tr_{d_1;t_1}(P;[u;v,u_s]) - \sum_{0< t_1\le d_1} {{t_1+t_2}-1 \choose t_2} \Tr_{d_1;t_1}(P;[u;v,u_s])\\
& \quad \ge 0.
\end{align*}
If $d_1 \ge sm$, then $\F_{d_1}(P)[u,u_s, \hat{v}] \ne \emptyset $,
which implying that $\Tr_{d_1;t_1}(P;[u;u_s, \hat{v}])>0$ for some positive integer $t_1$.
So, Eq. (\ref{Tr_uv}) holds strictly in this situation.
The result follows by Corollary \ref{Tra_cut_2}.
\end{proof}

\begin{cor}\label{comp_path2}
Let $e=\{u,v, w_1,\ldots,w_{m-2}\}$ be an edge, and let $\H=e(w_1,\ldots,w_p) \odot(\H_1(\tilde{w}_1), \ldots,\H_p(\tilde{w}_p))$, where  $1 \le p \le m-2$.
Let $\H^{e}_{r,s}=\H(u,v) \odot (P^m_r(\tilde{u}),P^m_s(\tilde{v})$, where $\tilde{u},\tilde{v}$ are respectively the pendent vertices of $P_{r}^m$ and $P_s^m$.
If $r \ge s \ge 1$, then $$\Tr_d(\H^{e}_{r,s}) \ge \Tr_d(\H^{e}_{r+1,s-1}),$$
with strict inequality if $d \ge (s+1)m$.
\end{cor}

\begin{proof}
The proof line is similar to that in Corollary \ref{comp_path}.
Let $P_{r+s+1}^m$ be the hyperpath with some vertices labelled and ordered from left to right as in Fig. \ref{path2}, whose edges are labelled from left to right as
$e_i=\{v_i,v_{i,1},\ldots,v_{i,m-2}, v_{i-1}\}$ for $i \in [s]$, $e'_1=\{v_0,v_{0,1},\ldots,v_{0,m-2},u_0\}$, and $e'_i=\{u_{i-2}, u_{i-2,1},\ldots,u_{i-2,m-2},u_{i-1}\}$ for $i=2,3,\ldots,r+1$.
Then $\H^{e}_{r,s}=P_{r+s+1}^m(v_{0,1},\ldots,v_{0,p}) \odot (\H_1(\tilde{w}_1),\ldots,\H_p(\tilde{w}_p))$, and $\H^{e}_{r+1,s-1}=P_{r+s+1}^m(v_{1,1},\ldots,v_{1,p}) \odot (\H_1(\tilde{w}_1),\ldots,\H_p(\tilde{w}_p))$.

\begin{figure}[h]
\centering
\includegraphics[scale=.8]{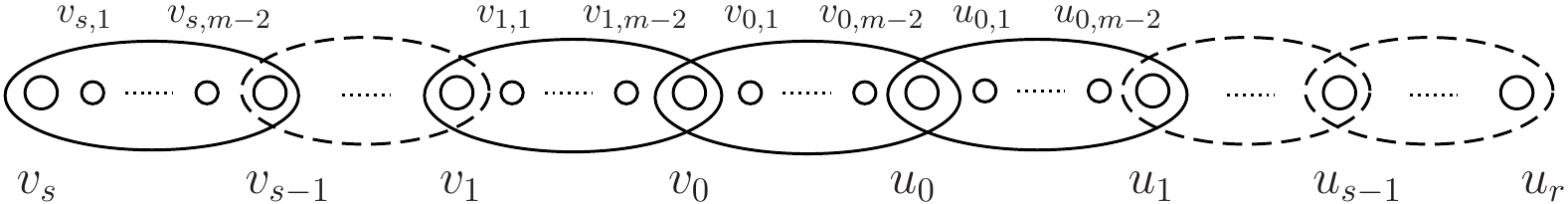}
\caption{The hyperpath $P_{r+s+1}^m$}\label{path2}
\end{figure}

Let $P:=P_{r+s+1}^m$, and
\begin{align*}
\tilde{\F}_d(\H^{e}_{r,s})[v_{0,1}]& =\{F \in \F^\e_d(\H^{e}_{r,s}): r_{F|_P}(v_{0,1})>0, r_{F|_{\H_1}}(\tilde{w}_1)+\cdots+r_{F|_{\H_p}}(\tilde{w}_p)>0\},\\
& =\bigcup_{d_0+d_1+\cdots+d_p=d}\bigcup_{t_i \le d_i, i=0,1,\ldots,p,\atop t_0>0, t_1+\cdots+t_p>0}
\F_{d_0,d_1,\ldots,d_p;t_0,t_1,\ldots,t_p}(\H^{e}_{r,s}),
\end{align*}
where
\begin{align*}
\F_{d_0,d_1,\ldots,d_p;t_0,t_1,\ldots,t_p}(\H^{e}_{r,s})& =\{F \in \tilde{\F}_d(\H^{e}_{r,s})[v_{0,1}]: |F|_P|=d_0, |F|_{\H_i}|=d_i,
r_{F|_P}(v_{0,1})=t_0, r_{F|_{\H_i}}(\tilde{w}_i)=t_i, i \in [p] \}.
\end{align*}
 We note here that if $r_{F|_P}(v_{0,1})=t_0>0$, then  $r_{F|_P}(v_{0,1})=\cdots=r_P(v_{0,p})=t_0$ by Corollary \ref{tree_root} as $\mathcal{V}(F)|_P$, the restriction of the Veblen hypergraph $\mathcal{V}(F)$ onto $P$, is still a Veblen hypergraph with an Euler rooting, and $v_{0,1},\ldots,v_{0,p}$ are cored vertices of $P$ contained in the same edge $e'_1$.
In the definition of $\tilde{\F}_d(\H^{e}_{r,s})[v_{0,1}]$, if replacing $v_{0,1}$ by $v_{1,1}$, we will have the definition of $\tilde{\F}_d(\H^{e}_{r+1,s-1})[v_{1,1}]$.
We get
\begin{equation}\label{tra_diff}
\Tr_d(\H^{e}_{r,s})-\Tr_d(\H^{e}_{r+1,s-1})=\sum_{F \in \tilde{\F}_d(\H^{e}_{r,s})[v_{0,1}]}\omega(F)-
\sum_{F \in \tilde{\F}_d(\H^{e}_{r+1,s-1})[v_{1,1}]}\omega(F).
\end{equation}



Suppose that $\H^{e}_{r,s}$ has $n$ vertices, $P$ has $n_0$ vertices, and $\H_i$ has $n_i$ vertices for $i \in [p]$.
Then $n=n_0+n_1+\cdots+n_p-p$.
By applying a similar discussion as Lemmas \ref{Tra_cut}, we have
\begin{align*}
\sum_{F \in \tilde{\F}_d(\H^{e}_{r,s})[v_{0,1}]}\omega(F)&=
\sum_{d_0+d_1+\cdots+d_p=d}\sum_{t_i \le d_i, i=0,1,\ldots,p,\atop t_0>0, t_1+\cdots+t_p>0}
\sum_{F\in \F_{d_0,d_1,\ldots,d_p;t_0,t_1,\ldots,t_p}(\H^{e}_{r,s})}
d(m-1)^n \frac{\tau(F)}{\prod_{x \in V(F)} d^+_F(x)}\\
&=\sum_{d_0+d_1+\cdots+d_p=d}\sum_{t_i \le d_i, i=0,1,\ldots,p,\atop t_0>0, t_1+\cdots+t_p>0}
\sum_{F\in \F_{d_0,d_1,\ldots,d_p;t_0,t_1,\ldots,t_p}(\H^{e}_{r,s})}
d(m-1)^n \prod_{i \in S_F} \frac{t_0t_i(m-1)}{t_0+t_i}\\
& \quad \cdot
\prod_{i \in S_F} \frac{\tau(F|_{\H_i})}{\prod_{x \in V(F|_{\H_i})} d^+_{F|_{\H_i}}(x)}\cdot \frac{\tau(F|_P)}{\prod_{x \in V(F|_P)}d^+_{F|_P}(x)}\\
&=\sum_{d_0+d_1+\cdots+d_p=d}\sum_{t_i \le d_i, i=0,1,\ldots,p,\atop t_0>0, t_1+\cdots+t_p>0}
\prod_{i \in S_F} \frac{t_i t_0}{d_i(t_0+t_i)} {t_0+t_i \choose t_0} \sum_{F \in \F^\e_{d_i;t_i}(\H_i,[\tilde{w}_i])}\frac{d_i(m-1)^{n_i}\tau(F|_{\H_i})}{\prod_{x \in V(F|_{\H_i})} d^+_{F|_{\H_i}}(v)}\\
& \quad \cdot \frac{d(m-1)^{n-n_0-\sum_{i \in S_F}(n_i-1)}}{d_0} \sum_{F \in \F^\e_{d_0;t_0}(P,[v_{0,1},\ldots, v_{0,p}])}\frac{d_0(m-1)^{n_0}\tau(F|_{P})}{\prod_{x \in V(F|_{P})} d^+_{F|_{P}}(v)}\\
&= \sum_{d_0+d_1+\cdots+d_p=d}\sum_{t_i \le d_i, i \in [p],\atop t_1+\cdots+t_p>0}\prod_{i=1}^p \frac{t_i}{d_i} \Tr_{d_i;t_i}(\H_i;[\tilde{w}_i])
\cdot \frac{d}{d_0}  \sum_{0<t_0\le d_0} \prod_{i=1}^p
{t_0+t_i-1 \choose t_0-1} \cdot \Tr_{d_0;t_0}(P;[v_{0,1}]),
\end{align*}
where $S_F:=\{i: i \in [p], t_i \ne 0\}$, and in the last equality, if $t_i=0$ or equivalently $d_i=0$ we have $\frac{t_i}{d_i}=1$ and $\Tr_{d_i;t_i}(\H_i;[\tilde{w}_i])=(m-1)^{n_i-1}$ in this case so that the equality holds.

Similarly, we have
$$
\sum_{F \in \tilde{\F}_d(\H^{e}_{r+1,s-1})[v_{1,1}]}\omega(F)=
\sum_{d_0+d_1+\cdots+d_p=d}\sum_{t_i \le d_i, i \in [p],\atop  t_1+\cdots+t_p>0}\prod_{i=1}^p \frac{t_i}{d_i} \Tr_{d_i;t_i}(\H_i;[\tilde{w}_i])
\cdot \frac{d}{d_0}  \sum_{0<t_0\le d_0}\prod_{i=1}^p
{t_0+t_i-1 \choose t_0-1} \cdot \Tr_{d_0;t_0}(P;[v_{1,1}]).
$$
So, to prove the difference in (\ref{tra_diff}) is greater than or equal to zero, it suffices to prove that for any given nonnegative intergers $t_1,\ldots,t_p$ with $t_1+\cdots+t_p>0$ and positive integers $d_0$,
\begin{equation}\label{comp2}
\sum_{0 < t_0 \le d_0}\prod_{i=1}^p{t_i+t_0-1 \choose t_i}\cdot \Tr_{d_0;t_0}(P;[v_{0,1}])\ge \sum_{0 < t_0 \le d_0} \prod_{i=1}^p{t_i+t_0-1 \choose t_i}\cdot \Tr_{d_0;t_0}(P;[v_{1,1}]).
\end{equation}

By a similar discussion as in Corollary \ref{comp_path},
define a ``reflection'' $\phi$ on $P$ with respect to the vertex $v_0$, that is,
$$ \phi=\prod_{i=1}^{s-1}(v_i u_{i-1})\cdot  \prod_{i=1}^{m-2}(v_{1,i}v_{0,i}) \cdot \prod_{i=2}^{s} \prod_{j=1}^{ m-2 }(v_{i,j} u_{i-2,j}).$$
We get a map  $\tilde{\phi}$ from $\F^\e_{d_0}(P)[v_{1,1},\hat{u}_s]$ to $\F^\e_{d_0}(P)[v_{0,1},\hat{u}_s]$ induced by $\phi$ and defined as in (\ref{tphi}).
We also find that for each $F \in \F^\e_{d_0}(P)[v_{1,1},\hat{u}_s]$,
$r_F(v_{1,1})=r_{\tilde{\phi}(F)}(v_{0,1})$, and $\tilde{\phi}$ is a bijection between $\F^\e_{d_0;t_0}(P)[v_{1,1};\hat{u}_s]$ and $\F^\e_{d_0;t_0}(P)[v_{0,1};\hat{u}_s]$.
By Eq. (\ref{Shaoeq_Hyp3}), we have
\begin{equation}\label{refl2}
\Tr_{d_0;t_0}(P;[v_{1,1};\hat{u}_s])=\Tr_{d_0;t_0}(P;[v_{0,1};\hat{u}_s]).
\end{equation}
Note that
$\F^\e_{d_0}(P)[v_{1,1},u_s]=\F^\e_{d_0}(P)[v_{1,1},v_{0,1},u_s]$ and
\begin{equation}\label{addi2} \F^\e_{d_0}(P)[v_{0,1},u_s]=\F^\e_{d_0}(P)[v_{0,1},u_s,\hat{v}_{1,1}]\cup \F^\e_{d_0}(P)[v_{0,1},u_s, v_{1,1}].
\end{equation}
By Corollary \ref{tree_root}, $\F^\e_{d_0}(P)[v_{0,1},u_s,\hat{v}_{1,1}] \ne \emptyset $  if $d_0 \ge (s+1)m$,
and $\F^\e_{d_0}(P)[v_{0,1},u_s, v_{1,1}] \ne \emptyset $  if $d_0 \ge (s+2)m$.
We will prove that if $d_0 \ge (s+2)m$,
\begin{equation}\label{comp3}
\sum_{0 < t_0 \le d_0} \prod_{i=1}^p{t_i+t_0-1 \choose t_i}\cdot \Tr_{d_0;t_0}(P;[v_{0,1};v_{1,1},u_s])
 \ge \sum_{0 < t_0 \le d_0} \prod_{i=1}^p{t_i+t_0-1 \choose t_i}\cdot \Tr_{d_0;t_0}(P;[v_{1,1};v_{0,1},u_s]),
 \end{equation}
 or equivalently
 \begin{align*}
&\sum_{F \in \F_{d_0}(P)[v_{0,1},v_{1,1},u_s]}\prod_{i=1}^p {t_i+r_F(v_{0,1})-1 \choose t_i}\cdot \frac{\tau(F)}{\prod_{x \in V(F)} d^+_F(x)} \\
& \quad \ge  \sum_{F \in \F_{d_0}(P)[v_{0,1},v_{1,1},u_s ]}\prod_{i=1}^p{t_i+r_F(v_{1,1}) -1 \choose t_i}\cdot \frac{\tau(F)}{\prod_{x \in V(F)} d^+_F(x)}.
\end{align*}

By Corollary \ref{tree_root}, if $ \F^\e_{d_0}(P)[v_{0,1},v_{1,1},u_s] \ne \emptyset$, then $m \mid d_0$; furthermore, we have a decomposition as
We have a decomposition
$$ \F^\e_{d_0}(P)[v_{0,1},v_{1,1},u_s]=\bigcup_{1\le q_v \le s,\atop s+1 \le q_u \le r+1}
\bigcup_{\sum\limits_{i=1}^{q_v} m_i+ \sum\limits_{i=}^{q_u}m'_i=\frac{d_0}{m}}\bigcup_{\sigma_i \in \mathbb{S}(\{m_i,m'_i\}),\atop
i \in [q_v]}
\F(\sigma_1(m_1,m'_1)\ldots,
\sigma_{q_v}(m_{q_v},m'_{q_v}),m'_{q_v+1},\ldots,m'_{q_u}),$$
where $\F(\sigma_1(m_1,m'_1)\ldots,
\sigma_{q_v}(m_{q_v},m'_{q_v}),m'_{q_v+1},\ldots,m'_{q_u})$ denotes the set of $d_0$-tuples of ordered rooted edges consisting of $e_i$ with multiplicity $m\sigma_i(m_i)$,  $e'_i$ with multiplicity $m\sigma_i(m'_i)$ for $i \in [q_v]$, and edges $e'_{j}$ with multiplicity $mm'_j$ for $j=q_v+1,\ldots,q_u$.

Given $q_v,q_u,m_1,m'_1,\ldots,m_{q_v},m'_{q_v}, m'_{q_v+1},\ldots,m'_{q_u}$, simply denote
$$\F(\sigma_1,\sigma_2,\ldots,\sigma_{q_v}):=\F(\sigma_1(m_1,m'_1)\ldots,
\sigma_{q_v}(m_{q_v},m'_{q_v}),m'_{q_v+1},\ldots,m'_{q_u}),$$
and
 \begin{align*}
&D(\sigma_1,\sigma_2,\ldots,\sigma_{q_v}) \\
&:=
\sum_{F\in \F(\sigma_1,\sigma_2,\ldots,\sigma_{q_v})}
  \left(\prod_{i=1}^p {t_i+r_F(v_{0,1})-1 \choose t_i}
 -\prod_{i=1}^p{t_i+r_F(v_{1,1}) -1 \choose t_i}\right)\frac{\tau(F)}{\prod_{x \in V(F)} d^+_F(x)}
.
\end{align*}
By a calculation as in Corollary \ref{comp_path}, we get
$$r_F(v_{0,p})=\s_1(m'_1), r_F(v_{1,p})=\s_1(m_1),$$
\begin{align*}
    \prod_{x \in V(F)}d_F^+(x)&=\s_{q_v}(m_{q_v})(\s_{q_v}(m'_{q_v})+m'_{q_v+1})\prod_{i=1}^{q_v-1}(\s_i(m_i)+\s_{i+1}(m_{i+1}))(\s_i(m'_i)+\s_{i+1}(m'_{i+1}))\\
    &\quad  \cdot (m_1+m'_1)\prod_{i=q_v+1}^{q_u-1} (m'_i+m'_{i+1}) \cdot m'_{q_u}\left(\prod_{i=1}^{q_v}m_i m'_i \prod_{i=q_v+1}^{q_u} m'_i\right)^{m-2} (m-1)^{(q_u+q_v)(m-1)+1},
\end{align*}
$$\tau(F)=m^{(q_u+q_v)(m-2)}\left(\prod_{i=1}^{q_v}m_i m'_i \prod_{i=q_v+1}^{q_u}m'_i\right)^{m-1}.$$
We also note that $\F(\sigma_1,\sigma_2,\ldots,\sigma_{q_v})$ contains exactly
$$ {(m_1+m'_1 \choose m_1} {\sigma_{q_v}(m'_{q_v})+m'_{q_v+1} \choose \sigma_{q_v}(m'_{q_v})}\prod_{i=1}^{q_{v}-1}{\sigma_i(m_i)+\sigma_{i+1}(m_{i+1}) \choose \sigma_i(m_i)}{\sigma_i(m'_i)+\sigma_{i+1}(m'_{i+1}) \choose \sigma_i(m'_i)} \cdot
 \prod_{i=q_v+1}^{q_u-1} {m'_{i}+m'_{i+1} \choose m'_i}
 $$
$d_0$-tuples of ordered rooted edges, all corresponding the same summand in $D(\s_1,\ldots,\s_{q_v})$.
So we have
\begin{align*}
 D(\s_1,\ldots,\s_{q_v})&=\alpha\left(\prod_{i=1}^p {t_i+\s_1(m'_1)-1 \choose t_i} -  \prod_{i=1}^p {t_i+\s_1(m_1)-1 \choose t_i} \right)\\
  & \quad \cdot {\sigma_{q_v}(m'_{q_v})+m'_{q_v+1}-1 \choose \sigma_{q_v}(m'_{q_v})-1}
 \prod_{i=1}^{q_{v}-1}{\sigma_i(m_i)+\sigma_{i+1}(m_{i+1})-1 \choose \sigma_i(m_i)-1}{\sigma_i(m'_i)+\sigma_{i+1}(m'_{i+1})-1 \choose \sigma_i(m'_i)-1},
 \end{align*}
where
$$ \alpha=\frac{
m^{(q_u+q_v)(m-2)} {m_1+m'_1 \choose m_1} \prod_{i=q_v+1}^{q_u-1}
{m'_{i}+m'_{i+1}-1 \choose m'_i-1}}{(m_1+m'_1)(m-1)^{(q_u+q_v)(m-1)+1}}.$$

For $2 \le i \le q_v$, denote
$$ D(\sigma_i,\ldots,\sigma_{q_v})
=\sum_{\sigma_1,\ldots,\sigma_{i-1}}D(\sigma_1,\sigma_2,\ldots,\sigma_{q_v}),
$$
and
$$ D_{\emptyset}=\sum_{\sigma_1,\ldots,\sigma_{q_v}}
D(\sigma_1,\sigma_2,\ldots,\sigma_{q_v}).
$$
We may assume that $m_i \ne m'_i$ for each $i \in q_v$;
otherwise, $D_\emptyset =0$ as shown in the below.

We have
\begin{align*}
    D(\s_2,\ldots,\s_{q_v})&= \alpha f_1(m_1,m'_1)(m'_1-m_1)\\
    & \quad \cdot
    \left({m_1+\s_2(m_2)-1 \choose m_1-1}
    {m'_1+\s_2(m'_2)-1 \choose m'_1-1}-{m'_1+\s_2(m_2)-1 \choose m'_1-1}
    {m_1+\s_2(m'_2)-1 \choose m_1-1}\right)\\
    &    \quad \cdot  {\sigma_{q_v}(m'_{q_v})+m'_{q_v+1} -1\choose \sigma_{q_v}(m'_{q_v})-1}
    \prod_{i=2}^{q_{v}-1}{\sigma_i(m_i)+\sigma_{i+1}(m_{i+1})-1 \choose \sigma_i(m_i)-1}{\sigma_i(m'_i)+\sigma_{i+1}(m'_{i+1})-1 \choose \sigma_i(m'_i)-1},
\end{align*}
where
$$f_1(m_1,m'_1):=\frac{\prod_{i=1}^p {t_i+m'_1-1 \choose t_i} -  \prod_{i=1}^p {t_i+m_1-1 \choose t_i} }{(m'_1-m_1)}.$$
Denote
$$f_i(m_{i-1},m'_{i-1},m_i,m'_i)=\frac{{m_{i-1}+m_i-1 \choose m_{i-1}-1}{m'_{i-1}+m'_i-1 \choose m'_{i-1}-1}-
{m'_{i-1}+m_i-1 \choose m'_{i-1}-1}{m_{i-1}+m'_i-1 \choose m_{i-1}-1}}{(m'_{i-1}-m_{i-1})(m'_i-m_i)},i=2,\ldots,q_v.$$
 By induction, we have
 \begin{align*}
     D(q_v)&= \alpha f_1(m_1,m'_1)
    \prod_{i=2}^{q_v-1}f_i(m_{i-1},m'_{i-1},m_i,m'_i)
    \prod_{i=1}^{q_v-2}(m'_i-m_i)^2(m'_{q_v-1}-m_{q_v-1})\\
     & \quad \cdot
     \left({m_{q_v-1}+\s_{q_v}(m_{q_v})-1 \choose m_{q_v-1}-1}
    {m'_{q_v-1}+\s_{q_v}(m'_{q_v})-1 \choose m'_{q_v}-1}-{m'_{q_v-1}+\s_{q_v}(m_{q_v})-1 \choose m'_{q_v-1}-1}
    {m_{q_v-1}+\s_{q_v}(m'_{q_v})-1 \choose m_{q_v-1}-1}\right)\\
        & \quad \cdot  {\sigma_{q_v}(m'_{q_v})+m'_{q_v+1}-1 \choose \sigma_{q_v}(m'_{q_v})-1}.
 \end{align*}
So,
\begin{align*}
    D_\emptyset&=D(1)+D((m_{q_v}m'_{q_v}))\\
    &=\alpha f_1(m_1,m'_1)\prod_{i=2}^{q_v}f_i(m_{i-1},m'_{i-1},m_i,m'_i) \cdot \prod_{i=1}^{q_v}(m'_i-m_i)^2\\
    & \cdot \frac{1}{(m'_{q_v}-m_{q_v})}\left({m'_{q_v}+m'_{q_v+1}-1 \choose m'_{q_v}-1}-{m_{q_v}+m'_{q_v+1}-1 \choose m_{q_v}-1}\right).
\end{align*}

In the above discussion, if there exists a smallest $i \in [q_v]$ such that
$m_i=m'_i$, then $D(\sigma_i,\ldots,\sigma_{q_v})=0$ and hence
$$D_\emptyset=\sum_{\sigma_i,\ldots,\sigma_{q_v}}D(\sigma_i,\ldots,\sigma_{q_v})=0.$$
So, if $m_i\ne m'_i$ for all $i$, then $f_i(m_{i-1},m'_{i-1},m_i,m'_i)>0$ for all $i$, implying that
$ D_\emptyset >0.$

We are now arriving at the inequality (\ref{comp3}) as
\begin{align*}
&\sum_{0 < t_0 \le d_0}  \prod_{i=1}^p{t_i+t_0-1 \choose t_i}\cdot \Tr_{d_0;t_0}(P;[v_{0,1};v_{1,1},u_s])  - \sum_{0 < t_0 \le d_0}  \prod_{i=1}^p{t_i+t_0-1 \choose t_i}\cdot \Tr_{d_0;t_0}(P;[v_{1,1};v_{0,1},u_s])\\
&\quad =\sum_{1 \le q_v \le s,\atop s+1 \le q_u \le r+1}
\sum_{\sum_{i=1}^{q_v} m_i+ \sum_{i=}^{q_u}m'_i=\frac{d_0}{m}}\sum_{\sigma_i \in \mathbb{S}(\{m_i,m'_i\}),i\in [q_v]}D(\sigma_1,\ldots,\sigma_{q_v})\\
&\quad  \ge 0.
\end{align*}
By Eqs. (\ref{refl2}), (\ref{addi2}) and (\ref{comp3}),
we proved Eq. (\ref{comp2}) as
\begin{align*}
    & \sum_{0 < t_0 \le d_0} \prod_{i=1}^p{t_i+t_0-1 \choose t_i}\cdot \Tr_{d_0;t_0}(P;[v_{0,1}])- \sum_{0 < t_0 \le d_0} \prod_{i=1}^p{t_i+t_0-1 \choose t_i}\cdot \Tr_{d_0;t_0}(P;[v_{1,1}])\\
    & =\sum_{0 < t_0 \le d_0}  \prod_{i=1}^p{t_i+t_0-1 \choose t_i}\cdot \Tr_{d_0;t_0}(P;[v_{0,1};u_s,\hat{v}_{1,1}])\\
    &\quad + \sum_{0 < t_0 \le d_0} \prod_{i=1}^p{t_i+t_0-1 \choose t_i}\cdot \Tr_{d_0;t_0}(P;[v_{0,1};v_{1,1},u_s])  - \sum_{0 < t_0 \le d_0}  \prod_{i=1}^p{t_i+t_0-1 \choose t_i}\cdot \Tr_{d_0;t_0}(P;[v_{1,1};v_{0,1},u_s])\\
&  \ge 0.
\end{align*}
If $d_0 \ge (s+1)m$, then $\F^\e_{d_0}(P)[v_{0,1},u_s,\hat{v}_{1,1}] \ne \emptyset $,
implying that $\Tr_{d_0;t_0}(P;[v_{0,1};u_s,\hat{v}_{1,1}])>0$ for some positive integer $t_0$.
So, (\ref{comp2}) holds strictly in this situation.
The result follows from Eq. (\ref{tra_diff}).
\end{proof}

\section{Estrada index of hypertrees}
In this section, we will determine the unique hypertree with minimum or maximum Estrada index among all hypertrees with fixed number of edges.
By Corollaries \ref{comp_path},  \ref{comp_path2} and Eq. (\ref{EE_tr}), we easily get the following lemmas.

\begin{lem}\label{Est_path}
Let $\H_{r,s}$ be defined as in Corollary \ref{comp_path}.
If $r \ge s \ge 1$, then
$$EE(\H_{r,s})> EE(\H_{r+1,s-1}).$$
\end{lem}

\begin{lem}\label{Est_path2}
Let $\H^e_{r,s}$ be defined as in Corollary \ref{comp_path2}.
If $r \ge s \ge 1$, then
$$EE(\H^e_{r,s})> EE(\H^e_{r+1,s-1}).$$
\end{lem}

\begin{lem}\label{comp_cutedge}
Let $e=\{u,v_1, \ldots,v_{m-1}\}$ be an edge, and let $\H_1=e(v_1,\ldots,v_p) \odot(\tilde{\H}_1(\tilde{v}_1), \ldots,\tilde{\H}_p(\tilde{v}_p))$, where  $1 \le p \le m-1$.
Then
$$\Tr_d(\H_1(v_1)\odot \H_2(w)) \ge \Tr_d(\H_1(u)\odot \H_2(w)),$$
with strict inequality if $d$ is a multiple of $m$, and hence
$$EE(\H_1(v_1)\odot \H_2(w)) > EE(\H_1(u)\odot \H_2(w)).$$

\end{lem}

\begin{proof}
We will prove that for all positive integers $d_1,t_2$,
\begin{equation}\label{Tr_uv_3}
\sum_{F \in \F^\e_{d_1}(\H_1)[v_1]} {{r_F(v_1)+t_2-1} \choose t_2} \frac{\tau(F)}{\prod_{x \in V(F)}d_F^+(x)} \ge \sum_{F \in \F^\e_{d_1}(\H_1)[u]} {{r_F(u)+t_2-1} \choose t_2} \frac{\tau(F)}{\prod_{x \in V(F)}d_F^+(x)}.
\end{equation}
Note that $\F_{d_1}(\H_1)[u]  \subseteq \F_{d_1}(\H_1)[v_1]$ as $u$ is a cored vertex of $\H_1$.
For each $F \in \F_{d_1}(\H_1)[u]$, the Veblen hypergraph $\mathcal{V}(F)|_e$ is still a
Veblen hypergraph, implying that $u$ acts as a root of $e$ in the same times as $v_1$ by Corollary \ref{tree_root}.
So
$ r_F(u) \le r_F(v_1)$.
Thus the inequality (\ref{Tr_uv_3}) holds.
If $d$ is a multiple of $m$, then $\F_{d_1}(\H_1)[v_1] \backslash \F_{d_1}(\H_1)[u]  \ne \emptyset$ as $v_1$ is also contained in the edges of $\tilde{\H}_1$ except $e$.
So the inequality (\ref{Tr_uv_3}) holds strictly in this case.
The result follows by Corollary \ref{Tra_cut_2}.
\end{proof}

\begin{thm}
Let $T$ be a hypertree with $z$ edges.
Then
$$ EE(P_z^{m}) \le EE(T) \le EE(S_z^m),$$
with left equality if and only if $T=P_z^{m}$ and right equality if and only if $T=S_z^m$.
\end{thm}

\begin{proof}
Let $T_0$ be a hypertree with minimum Estrada index among all hypertrees with $z$ edges.
If $T_0$  is not  a hyperpath, then $\T_0$ can be obtained from a sub-hypertree $T_1$ of $T_0$ by attaching two paths $P_r^m, P_s^m$ with their pendent vertices identified with $u,v$ of $T_1$ respectively, where $u,v$ belong to the same edge say $e$ of $T_1$, and $r \ge s \ge 1$.
If $u=v$, then $T_0$ can be written as $T_0=:\H_{r,s}=T_1(u) \odot ( P_{r}^m(u), P_s^m(u))$.
By Lemma \ref{Est_path}, we have $$EE(T_0)=EE(\H_{r,s})>EE(\H_{r+1,s-1}),$$
which yields a contradiction to the definition of $T_0$.
If $u \ne v$, then $T_0$ can be written as $T_0=:\H^e_{r,s}=T_1(u,v) \odot (P_{r}^m(u),P_s^m(v))$.
By Lemma \ref{Est_path2}, we have $$EE(T_0)=EE(\H^e_{r,s})>EE(\H^e_{r+1,s-1}),$$
which also yields a contradiction.
So $T_0=P_z^m$.

Next let $T'_0$ be a hypertree with maximum Estrada index among all hypertrees with $z$ edges.
Suppose that $T'_0$ is not a hyperstar.
Let $u$ a vertex of $T'_0$ with maximum degree $\Delta$.
Note that $\Delta \ge 2$.
Then $T'_0$ is obtained from the hyperstar $S_\Delta^m$ centered at $u$ by attaching some sub-hypertrees to the vertices of  $S_\Delta^m$ except the vertex $u$.
So $T'_0$ can be written as $T'_0=T'_1(v) \odot \T'_2(v)$, where $T'_1$ contains $S_\Delta^m$ as a sub-hypertree, and $v, u$ belong to the same edge of $S_\Delta^m$.
By Lemma \ref{comp_cutedge}, we have
$$EE(T'_0)=EE(T'_1(v) \odot T'_2(v)) < EE(T'_1(u) \odot T'_2(v)),$$
which  yields a contradiction to the definition of $T'_0$.
So $\T_1=S_z^m$.
\end{proof}

\end{document}